\documentclass[letterpaper, 10 pt, conference]{ieeeconf}  

\IEEEoverridecommandlockouts                              

\usepackage{cite}
\usepackage{amsmath,amssymb,amsfonts}
\usepackage{graphicx}
\usepackage{caption,subcaption}
\usepackage{algorithm,algorithmicx}
\usepackage{hyperref}
\usepackage{color}

\usepackage{mathtools}
\usepackage[nolist]{acronym}

\newtheorem{theorem}{Theorem}
\newtheorem{corollary}{Corollary}
\newtheorem{lemma}{Lemma}
\newtheorem{problem}{Problem}
\newtheorem{remark}{Remark}
\newtheorem{example}{Example}

\newtheorem{proposition}{Proposition}
\newcommand{\beq}{\begin{equation}}
\newcommand{\eeq}{\end{equation}}

\newcommand{\E}{\mathbb{E}}
\newcommand{\R}{\mathbb{R}}

\DeclarePairedDelimiter{\Tr}{\text{Tr}(}{)}
\newcommand{\st}{\mathop{s.t.}}

\usepackage{bm}
\usepackage{bbm}

\title{\LARGE \bf
Robust Data-Driven Receding-Horizon Control for LQR with Input Constraints
}

\author{Jian Zheng and Mario Sznaier
\thanks{This work has been submitted to IEEE L-CSS for possible publication.}%
\thanks{This work was partially supported by NSF grant CMMI 2208182, AFOSR grant FA9550-19-1-0005 and  ONR grant N00014-21-1-2431.}%
\thanks{J. Zheng and M. Sznaier are with the Robust Systems Lab,  ECE Department, Northeastern University, Boston, MA 02115.
        {\tt\small (e-mails: zheng.jian1@northeastern.edu, msznaier@coe.neu.edu)}}%
}

\begin{document}

\maketitle
\thispagestyle{empty}
\pagestyle{empty}

\begin{abstract}

This letter presents a robust data-driven receding-horizon control framework for the discrete time  linear quadratic regulator (LQR)  with input constraints. Unlike existing data-driven approaches that design a controller from initial data and apply it unchanged throughout the trajectory, our method exploits all available execution data in a receding-horizon manner, thereby capturing additional information about the unknown system and enabling less conservative performance. Prior data-driven LQR and model predictive control methods largely rely on Willem's fundamental lemma, which requires noise-free data, or use regularization to address disturbances, offering only practical stability guarantees. In contrast, the proposed approach extends semidefinite program formulations for the data-driven LQR to incorporate input constraints and leverages duality to provide formal robust stability guarantees. Simulation results demonstrate the effectiveness of the method.

\end{abstract}

\section{INTRODUCTION}

The \ac{LQR} is a benchmark problem in control theory, where the objective is to minimize an infinite-horizon quadratic cost for \ac{LTI} systems, with the optimal gain obtained from solving the algebraic Riccati equation \cite{khalil1996robust, anderson2007optimal}. A well-known extension is \ac{MPC}, which optimizes a finite-horizon cost while explicitly handling state and input constraints, making it widely adopted in practice \cite{rawlings2020model}. However, \ac{MPC} critically depends on accurate nominal models to predict future states and solve online optimization problems. Traditional approaches obtain such models through system identification followed by robust control synthesis \cite{formentin2014comparison}, but identification can be costly and introduce modeling errors that may degrade closed-loop performance \cite{dai2023data}. 

Recent advances in computation and data availability have sparked growing interest in \ac{DDC}, where controllers are designed directly from measured data without explicit identification \cite{dai2020semi, de2019formulas, zheng2023robust, zheng2024data}. Most existing \ac{DDC} methods design controllers offline from initial measurements and apply them unchanged throughout the trajectory. While the designed controllers guarantee stability over the region of interest, they may be conservative in performance, e.g., convergence speed, since they do not exploit online execution data, which obtain additional information about the unknown system \cite{zheng2025receding}. 

Several data-driven counterparts of \ac{LQR} and \ac{MPC} have been developed \cite{van2020data, dorfler2022role, coulson2019data, coulson2019regularized, berberich2020data, berberich2022combining}. Data-driven \ac{LQR} methods \cite{van2020data, dorfler2022role} suffer from the same limitation above and cannot accommodate constraints. In contrast, data-enabled predictive control (DeePC) \cite{coulson2019data, coulson2019regularized} and data-driven predictive control \cite{berberich2020data, berberich2022combining} can handle constraints and inherently operate in a receding-horizon (RH) fashion. However, they update only a fixed-length window of the most recent data, rather than using all available measurements, including initial and past executions, which again limits performance and introduces conservatism. Recent work has attempted to address this by updating the consistency set (i.e., the set of systems consistent with the data) along the trajectory and enforcing set-membership via duality to achieve robust stability \cite{zheng2025receding}. While effective, this approach yields a greedy controller that optimizes the one-step convergence rate in the worst case, limiting long-term performance.

A common feature of the methods above is their reliance on Willem's fundamental lemma \cite{willems2005note}, which requires persistently exciting and noise-free data, restricting its applicability to \ac{LTI} systems without disturbances. Robust variants introduce regularization \cite{berberich2020data} or low-rank approximation \cite{dorfler2022role} to address disturbances, but they require case-dependent weight tuning and provide only practical, rather than formal, guarantees of robust stability. To overcome these limitations, \ac{SDP}-based formulations of data-driven \ac{LQR} have recently been introduced \cite{dai2023data, hu2025robust}. The work in \cite{dai2023data} extends the \ac{SDP} framework for \ac{LQR} control \cite{feron1992numerical} and provides formal robustness guarantees against $\ell_\infty$-bounded noise, serving as the baseline for this paper, but it cannot handle input constraints. In contrast, \cite{hu2025robust} addresses finite-horizon \ac{MPC} with constraints by formulating a consistency set and enforcing set-membership through the S-lemma \cite{van2020noisy}. However, the consistency set is approximated by an outer ellipsoid, 
which introduces conservatism and may limit performance. 

Motivated by these challenges, this paper develops a robust data-driven receding-horizon control framework for \ac{LQR} problems with input constraints and $\ell_\infty$-bounded noise.  The proposed method extends prior \ac{SDP}-based approaches by explicitly incorporating input constraints and dynamically updating the consistency set using all available data, thereby tightening the worst-case \ac{LQR} performance bound. Unlike existing approaches, it provides formal robust stability guarantees while reducing conservatism. The main contributions are as follows:

\begin{enumerate}
\item A less conservative, computationally efficient  \ac{SDP}-based  receding-horizon \ac{LQR} framework capable of  explicitly handling input constraints through the introduction of a parameter $\tau$ that scales the control action while keeping it close to its bounds when needed.
    \item Extension of the framework to the data-driven case. 
    
    \item Substantial reduction in the computational complexity of this data-driven framework by exploiting duality.
\end{enumerate}

The rest of the paper is organized as follows. Section~\ref{sec:preliminaries} presents preliminaries and problem formulation. Section~\ref{sec:model-based} and \ref{sec:data-driven} propose the \ac{SDP}-based input-constrained \ac{LQR} control algorithms in the model-based and data-driven receding-horizon settings, respectively. 
Section \ref{sec:experiments} presents simulation results that demonstrate the effectiveness of the proposed method, and Section \ref{sec:conclusions} concludes the paper.

\section{PRELIMINARIES} \label{sec:preliminaries}

\begin{acronym}
\acro{DDC}{data-driven control}
\acro{LMI}{linear matrix inequality}
\acroplural{LMI}[LMIs]{linear matrix inequalities}
\acroindefinite{LMI}{an}{a}
\acro{LP}{linear program}
\acroindefinite{LP}{an}{a}
\acro{PSD}{positive semidefinite}
\acro{SDP}{semidefinite program}
\acroindefinite{SDP}{an}{a}
\acro{SOS}{sum-of-squares}
\acroindefinite{SOS}{an}{a}
\acro{LTI}{linear time-invariant}
\acro{LQR}{linear quadratic regulator}
\acro{MPC}{model predictive control}
\acro{P-satz}{Positivstellensatz}
\acro{RH}{receding horizon}
\end{acronym}

\subsection{Notations}
Let $\R$, $\R^n$ and $\R^{n\times m}$ denote the set of real numbers, $n$-dimensional real vectors and $n\times m$-dimensional real matrices. $\E[\cdot]$ and $\Tr{\cdot}$ denote the expectation and trace of a matrix. For a matrix $M$,  $M^T$ denotes its transpose, and the symbol $*$ denotes the symmetric counterpart of a block in a symmetric matrix. For a symmetric matrix $P$, the notation $P\succ 0(P\succeq 0)$ indicates that $P$ is positive definite (positive semidefinite). $I_n$ and $0_n$ denote the $n\times n$ identity and zero matrices, and $\mathbf{1}$ denotes all-ones vector of proper dimension. $\|\cdot\|_2$ and $\|\cdot\|_\infty$ denote the $\ell_2$- and $\ell_\infty$-norms of a vector. The set of real and \ac{SOS} polynomials in the indeterminate $x$ is denoted by $\R[x]$ and $\Sigma[x]$, respectively. The set of interior points of a set $\mathcal{S}$ is denoted by int$(\mathcal{S})$.

\subsection{Model-Based LQR}
Consider a discrete-time \ac{LTI} system 
\beq \nonumber
    x_{k+1} = A x_k + B u_k + w_k,
\eeq
where $x \in \mathbb{R}^n, u \in \mathbb{R}^m,$ and $w \in \mathbb{R}^n$ denote the state, input, and a stochastic disturbance, respectively. As standard, we will assume that $w_k$ are i.i.d. random variables with $\mathbb{E}(w_k)=0,\mathbb{E}(w_k w_k^T)=I_n$ and $\mathbb{E}(x_k w_k^T)=0_n$. The \ac{LQR} problem admits several equivalent formulations, either deterministic or stochastic. In this paper, we adopt  the following stochastic formulation \cite[Ch.~6]{kwakernaak1972linear}:  design a state-feedback controller that stabilizes the system while minimizing  the expected value of an average quadratic cost with state and input weight matrices $Q\succeq 0, R\succ 0$:
\begin{equation}\label{eq:stochastic_LQR}
    \min_{u} \quad \mathbb{E}_w \left \{
    \lim_{N\to \infty}\frac{1}{N}\left[\sum^{N-1}_{k=0} x_k^TQx_k + u_k^T R u_k\right] \right \}
\end{equation}
 It is well known that, in the absence of constraints, the optimal control law is a constant state-feedback $u_k = Kx_k$ which can be obtained from the following \ac{SDP} \cite{feron1992numerical}:
\begin{equation} \label{eq:h2}
\begin{split}
    \min_{\gamma, P, M, W} \quad & \gamma\\ 
    \st \quad \quad & C(\gamma, P, M, W)
    \succeq 0,\; P \succ 0,
\end{split}
\end{equation}
where $C=\text{diag}(C_1, C_2, C_3)$ and 
\begin{equation} \nonumber
\begin{split}
    C_1 & = \gamma -\Tr{QP}-\Tr{W} \\
    C_2 & = \begin{bmatrix}
        P-I & AP+BM \\ * & P
    \end{bmatrix} \\
    C_3 & = \begin{bmatrix}
        W & R^{1/2}M \\ * & P
    \end{bmatrix}.
\end{split}
\end{equation}
If \eqref{eq:h2} is feasible, the optimal feedback gain $K=MP^{-1}$ achieves the minimum \ac{LQR} cost $\gamma$. The closed-loop system is then stable as certified by the Lyapunov function $x^TP^{-1}x$.  

\begin{proposition}[\hspace{1sp}\cite{feron1992numerical}]
    The program \eqref{eq:h2} is well-posed if $R\succ 0$ and the system is controllable.
\end{proposition}


\subsection{Semialgebraic Optimization via Sum-of-Squares}

A basic semialgebraic set is defined by a finite collection of polynomial equalities and inequalities. In this paper, we restrict our focus to semialgebraic sets defined solely by polynomial inequalities:
\begin{equation} \nonumber
    \mathcal{S} = \left\{x\in\R^n \colon f_i(x)\geq 0,\; i=1,\ldots, m \right\}.
\end{equation}
A set $\mathcal{S}$ is said to be \textit{Archimedean} if there exists a finite $R$ and \ac{SOS} polynomials $\sigma_i \in\Sigma[x]$ such that
\begin{equation} \nonumber
    R - \|x\|_2^2 = \sigma_0 (x) + \sum\nolimits \sigma_i f_i(x).
\end{equation}
Putinar's \ac{P-satz} \cite{putinar1993positive} states that a polynomial $p(x)$ is positive over an Archimedean set $\mathcal{S}$ if and only if $p(x)$ admits a representation as 
\begin{equation}\nonumber
    p(x) = \sigma_0 (x) + \sum\nolimits \sigma_i f_i(x).
\end{equation}

\subsection{Duality}
Consider a primal problem of the form
\beq
\begin{split}
    p^* = \min_{x} \; & f_0(x)\quad \st \; f_i(x) \leq 0,\; i=1,\dots,m,
\end{split}
\eeq
where all $f_i$ are convex. The corresponding dual problem is
\beq
\begin{split}
    d^* = \min_{\lambda\geq 0} \; & \inf_x f_0(x) + \sum_{i=1}^m \lambda_i f_i(x).
\end{split}
\eeq
\begin{lemma}[\hspace{1sp}\cite{boyd2004convex}] \label{lem:duality}
    If Slater's condition holds, i.e., there exists an $x$ such that $f_i(x)<0$ for all $i=1,\dots,m$, then strong duality holds with $p^*=d^*$.
\end{lemma}

\subsection{Problem Formulation}
Consider a discrete-time, single-input\footnote{Extension to multi-input systems is trivial with an $\ell_\infty$ control bound.} \ac{LTI} system with input constraints
\begin{equation} \label{eq:dynamics}
    x_{k+1} = A x_k + B u_k + w_k, \quad x_k, w_k \in \mathbb{R}^n, \; |u| \leq 1.
\end{equation}
 The system parameters $A,B$ are unknown and the disturbance is bounded by $\|w\|_\infty \leq \epsilon$ with no specific distribution assumed. Then the problem of interest to this paper is:
\begin{problem}\label{prob:lqr}
Given noisy data $\mathcal{D}=\left\{(x_k^d, u_k^d, x_{k+1}^d)\right\}^{N_d-1}_{k=0}$, find a controller $u(x)$ that renders the unknown system \eqref{eq:dynamics} stable and minimizes the quadratic cost function \eqref{eq:stochastic_LQR} subject to the constraint $|u|\leq 1$.
\end{problem}

\section{Receding-Horizon Constrained LQR} \label{sec:model-based}

In this section, we first show how, in the known model case  the \ac{SDP} problem \eqref{eq:h2} can be adapted to incorporate input constraints in a receding-horizon fashion. 
Let $x_k$ denote the current state of the system.
Consider the following \acp{LMI}, parametric in $x_k$:
\begin{subequations}
\label{eq:DLMI}
\begin{align}
\gamma_k - \Tr{QP_k} -\Tr{W_k} & \geq  0 \label{eq:DLMImodel1}\\
\begin{bmatrix}P_k-(\frac{1}{\tau_k}+\epsilon)I & AP_k+ BM_k \\ * & P_k \end{bmatrix} & \succeq 0 \label{eq:DLMImodel2}\\
\begin{bmatrix} \beta_k & M_k \\ * & P_k \end{bmatrix} & \succeq 0  \label{eq:DLMImodel3} \\
\begin{bmatrix} \beta_k & \beta_k x^T_k \\ * & P_k \end{bmatrix} & \succeq 0  \label{eq:DLMImodel4} \\
\begin{bmatrix}W_k & R^{\frac{1}{2}}M_k \\ * & P_k \end{bmatrix} & \succeq 0  \label{eq:DLMImodel5}
\end{align}    
\end{subequations}
and the receding-horizon control law shown in Algorithm \ref{alg:RH}:
\begin{algorithm}[h]
    \caption{Model-Based Input-Constrained LQR}\label{alg:RH}
\begin{algorithmic}[1]
    \State {Given initial condition $x_1$ and initial $P_0\succ 0,\tau_0>0$}
    \State Start at $k=1$, repeat:
    \begin{align}
    (P_k,M_k,W_k,\tau_k,\beta_k) \gets& 
    \mathop{\arg\min} \; \tau \gamma \nonumber \\ 
    \st \quad 
    & \eqref{eq:DLMI} \text{ holds, and} \nonumber \\
    & \tau_k \, \leq  \tau_{k-1} \label{eq:tau} \\
    & P_k \succeq P_{k-1} \label{eq:p}
    \end{align}
   \State \quad Apply the control $u_k \gets M_kP_k^{-1} x_k$
   \State \quad Collect new measurement $x_{k+1}$
   \State \quad $k \gets k+1$
   \State Until termination
\end{algorithmic}
\end{algorithm}

\begin{theorem}\label{thm:model_based} 
If $A$ is Schur stable, then the  inequalities \eqref{eq:DLMI}-\eqref{eq:p} are feasible for all $k \geq 0$, regardless of the initial condition $x_0$, and the control law generated by Algorithm \ref{alg:RH} has the following properties: 
    \begin{enumerate}
        \item It satisfies the constraint $|u_k| \leq 1$.
        \item It renders the origin an exponentially stable equilibrium point of the closed-loop system.
        \item It minimizes an upper bound of the \ac{LQR} cost.
    \end{enumerate}
\end{theorem}

\begin{proof} The proof is structured in four parts.

\noindent      \textbf{Asymptotic Stability:} Feasibility of  \eqref{eq:DLMImodel2} and \eqref{eq:p}  at $k-1$ implies that{, for $x_{k-1}\neq 0$,}  
\beq \label{eq:decreasing} \begin{aligned}
&V(x;k) \doteq x_{k}^TP_{k}^{-1}
x_k^T  \leq x_{k}^TP_{k-1}^{-1}
x_k^T \\
&= x_{k-1}^T(A+BM_{k-1}P_{k-1}^{-1})^TP_{k-1}^{-1}(A+BM_{k-1}P_{k-1}^{-1})x_{k-1} \\
&< x_{k-1}^TP_{k-1}^{-1}
x_{k-1}.
\end{aligned}
\eeq
Hence $V(x;k)=x^T P_k^{-1}x $ serves as a Lyapunov function, ensuring asymptotic stability.

\noindent \textbf{Feasibility:} Let $x_0$ denote the initial condition of the system. Since $A$ is stable, take $M_0=0$ in \eqref{eq:DLMImodel2} and select $P_0$ as a solution to the Lyapunov inequality
    \beq\nonumber
    P_0-AP_0A^T \succeq \Big(\frac{1}{\tau_0}+\epsilon\Big)I,
    \eeq
    where $\tau_0 > 0$ is arbitrary. Choose $\beta_0$ such that $1/\beta_0 \geq x_0^TP_0^{-1}x_0$. With these choices, \eqref{eq:DLMI} is feasible at $k=0$.  {Then for any subsequent time $k$,} feasibility of \eqref{eq:DLMI} at time $k-1$, together with \eqref{eq:decreasing}, implies $1/\beta_{k-1}> x_{k}^TP_{k}^{-1}x_k$. Hence $(P_{k-1},M_{k-1},W_{k-1},\tau_{k-1},\beta_{k-1})$ is also a feasible solution to \eqref{eq:DLMI} at time $k$.

\noindent     \textbf{Constraint Satisfaction:} From \eqref{eq:DLMImodel3} and  \eqref{eq:DLMImodel4}  we have $\beta_k\geq M_kP_k^{-1}M_k^T$ and $1/\beta_k \geq x_k^T P_k^{-1}x_k$. Applying the Cauchy-Schwarz inequality yields:
    \begin{equation*}
        |u_k|=|M_kP_k^{-1}x_k| \leq (M_kP_k^{-1}M_k^T)^\frac{1}{2}(x_k^T P_k^{-1}x_k)^\frac{1}{2} \leq 1.
    \end{equation*}
     
\noindent       \textbf{Minimization of Upper Bound on LQR Cost:} Conditions \eqref{eq:DLMImodel1} and \eqref{eq:DLMImodel5} guarantee that $\tau\gamma$ is an upper bound of the \ac{LQR} cost under the control action $u_k$. 
\end{proof}


    
     
     

\begin{remark}
    Note that equations \eqref{eq:DLMI} are \acp{LMI} in the variables $(P,M,W,\frac{1}{\tau},\beta)$, and the objective is to minimize $\tau \gamma$. Since $\tau \geq 0$ is a scalar and is monotonically decreasing,  the minimization in the algorithm can be carried out by performing a line search on $\tau$. 
\end{remark}

\begin{corollary} \label{cor}Assume that $A$ is unstable but \eqref{eq:DLMI} is feasible for some initial condition $x_0$. Then, Algorithm \ref{alg:RH} asymptotically stabilizes the system in the region $\mathcal{R} \doteq \{x: x^TP_0^{-1}x \leq x_0^TP_0^{-1}x_0\}$.
\end{corollary}
\begin{proof} This follows from the fact, shown in the proof of  Theorem \ref{thm:model_based}, that once the algorithm enters a feasible region, it remains feasible thereafter. \end{proof}

\begin{remark}
The hypothesis of open-loop stability is necessary for global stabilization. It is well known that, generically, systems with poles {outside of the unit circle} 
cannot be globally stabilized with bounded controls. If this hypothesis fails, Corollary \ref{cor} shows that local stabilization is still achieved. However, in this case the region of attraction is given implicitly, in terms of the feasibility of \eqref{eq:DLMI}. A similar situation arises  in \cite{berberich2022combining, hu2025robust} where the region of attraction is also characterized in terms of feasibility of a set of inequalities.
\end{remark}

\begin{remark} Note that if \eqref{eq:DLMImodel2} is feasible for $(P,M,\tau,\epsilon)$ it is also feasible for
$(P_\tau, \frac{M}{\tau},1,\tau \epsilon)$, where $P_\tau=\tau P$. Thus the control action generated from the solutions of \eqref{eq:DLMImodel2} can be interpreted as:
$u = \tau^{-1}{M}P_{\tau}^{-1}x =  \tau^{-1}u_{unconst}$, where $u_{unconst}$ is the solution to the unconstrained LQR problem ($\tau=1$). Hence $\tau$ plays the role of a scaling parameter that guarantees that $u_k$ is saturated on the boundary of the ellipsoid $x_{k}^TP_{k}^{-1}x_k$.
\end{remark}

\section{Data-Driven Receding-Horizon LQR} \label{sec:data-driven}

   In this section, we extend our framework to case where  the system dynamics are unknown.  Rather than pursuing a classical identification/robust control pipeline, we adopt a data-driven, set-membership approach. Let $\mathcal{D}_{train}=\left\{(x^d_k, u^d_k, x^d_{k+1})\right\}_{k=0}^{N_d-1}$ denote  experimental training data  and 
    $\mathcal{D}_{N}=\left\{(x_k, u_k,x_{k+1})\right\}_{k=0}^{N-1}$ denote data generated during the execution up to time $N$.
   At each time instant, we first construct the tightest set $\mathcal{C}_N$ that contains the unknown dynamics, based on all available data, $\mathcal{D}=\mathcal{D}_{train}\cup \mathcal{D}_{N}$. Next, we find and apply an admissible control law $u_N$ such that  the LMIs \eqref{eq:DLMI} hold for the current state $x_N$. Finally, we update $N \gets N+1$, $x_N \gets x_{N+1}$ and repeat the process. Pursuing a RH scheme improves performance vis-a-vis non-RH approaches  by handling the control constraint less conservatively and by using the tightest description of the uncertain plant.  In the remainder of this section, we show that this approach indeed minimizes a worst-case bound on the LQR cost, and address computational complexity issues.

\subsection{Consistency Set}

Each noisy data sample $(x_k, u_k, x_{k+1})$ imposes linear constraints on the system parameters $(A,B)$. Thus, the set $\mathcal{C}_N$ is a \emph{polytope}, defined  by $N+N_d$ sets of constraints: 
\begin{align} \label{eq:consistency}
\mathcal{C}_N = \left\{ \vphantom{\begin{bmatrix} A & B \\ -A & -B \end{bmatrix}}
A,B \colon\right. & 
\begin{bmatrix}
    A & B \\ -A & -B
\end{bmatrix}
\begin{bmatrix}
    x_k \\ u_k
\end{bmatrix} \leq
\begin{bmatrix}
    \epsilon \mathbf{1} + x_{k+1} \\ 
    \epsilon \mathbf{1} - x_{k+1}
\end{bmatrix}, \nonumber \\
& \forall (x_k,u_k,x_{k+1}) \in \mathcal{D}
\left.\vphantom{\begin{bmatrix} A & B \\ -A & -B \end{bmatrix}} \right\}.
\end{align}

\subsection{Conceptual Robust Data-Driven LQR}
From Theorem~\ref{thm:model_based}, a feasible data-driven controller must satisfies \eqref{eq:DLMI} for all systems in the consistency set \eqref{eq:consistency}.  Thus, a RH robust control law that minimizes an worst-case bound on the LQR cost  can be found by solving, at each time step, a problem of the form

\begin{problem} \label{prob:robust_lqr}
Given data samples $\mathcal{D}$, 
find $\tau, \beta,\gamma > 0$ and matrices $P\in\R^{n\times n}\succ 0, M\in\R^{1\times n}, W\in\R$ that solve: 
\begin{align} \label{eq:robust_lqr}
    \min_{P, M, W, \tau, \beta,\gamma} & \max_{A,B} \quad \tau \gamma \\
    \st \;\;\quad & \textrm{Equations \eqref{eq:DLMI}, \eqref{eq:consistency} hold.} \nonumber
\end{align}
\end{problem}

Since the problem above is affine in $(A,B)$ and the consistency set is a polytope, it can be theoretically solved by \ac{SDP} solvers through a line search on $\tau$ at each vertex of the consistency set \cite{chesi2010lmi}. However, this approach is impractical as the number of vertices in the worst case grows combinatorially with the number of samples \cite{nesterov1994interior}. Alternatively, since \eqref{eq:DLMI} are affine in the design parameters and $\mathcal{C}_N$ is semialgebraic, the problem can be relaxed into an \ac{SDP} using Scherer's P-satz \cite{scherer2006matrix}. However, as shown in \cite{dai2023data}, this will involve matrices of size $\mathcal{O}(n^{2r})$, where $r$ is the order of the relaxation. Thus, it becomes impractical for moderately large systems, even if $r=1$. As we show next, computational complexity can be substantially reduced through the use of duality.

\subsection{Reducing Computational Complexity via Duality}
\begin{theorem} \label{thm:ddc} Assume that int$(\mathcal{C}_N) \not = \emptyset$. Then Problem~\eqref{prob:robust_lqr} is feasible if and only if there exist $\tau, \beta,\gamma > 0$, matrices $P\in\R^{n\times n}\succ 0, M\in\R^{1\times n}, W\in\R$ and a   matrix $\Lambda(z)$ with continuous entries $\Lambda_{i,j}(z) \geq 0$, $i=1,\ldots, 2(N+N_d)$, $j=1,\ldots, n$ in $\|z\|_2 \leq 1$ such that:
    \begin{align} \label{eq:thm2_1}
    \Theta \Lambda(z) &= -2\begin{bmatrix}
        P \\ M
    \end{bmatrix}\begin{bmatrix}
        I_n & 0_n
    \end{bmatrix} z_1 z_1^T \begin{bmatrix}
        0_n \\ I_n
    \end{bmatrix} \\
    0 & < g(z)- \Tr{D\Lambda(z)}, \label{eq:thm2_2}
    \end{align}
where $z=[z_1^T, z_2^T, z_3^T, z_4^T, z_5^T]^T$ and 
    \begin{align*}
    g(z) & = z_1^T \begin{bmatrix} P-(\frac{1}{\tau}+\epsilon)I &  0 \\ * & P \end{bmatrix} z_1 + 
    z_2^T\begin{bmatrix} \beta & M \\ * & P \end{bmatrix}z_2 \\
    &\quad +z_3^T\begin{bmatrix} \beta & \beta x^T \\ * & P \end{bmatrix} z_3 +
    z_4^T\begin{bmatrix}W & R^{\frac{1}{2}}M \\ * & P \end{bmatrix} z_4\\
    &\quad + z_5^T(\gamma - \Tr{QP} -\Tr{W})z_5,\\
D&= \begin{bmatrix}D_{train}& D_N \end{bmatrix}, \quad  \Theta = \begin{bmatrix}\Theta_{train}& \Theta_N \end{bmatrix}  \\
      D_{train} & = \begin{bmatrix}
            \epsilon \textbf{1} + x^d_1, \cdots, \epsilon \textbf{1} + x^d_{N_d}, & \epsilon \textbf{1} - x^d_1, \cdots, \epsilon \textbf{1} - x^d_{N_d} 
        \end{bmatrix}\\
      D_N & = \begin{bmatrix}
            \epsilon \textbf{1} + x_1, \cdots, \epsilon \textbf{1} + x_N, & \epsilon \textbf{1} - x_1, \cdots, \epsilon \textbf{1} - x_N \\
        \end{bmatrix} \\
    \Theta_{train} &=  \begin{bmatrix}
        x_1^d, \cdots, x^d_{N_d}, & -x_1^d, \cdots, -x^d_{N_d} \\
        u^1_d, \cdots, u^d_{N_d}, & -u_1^d, \cdots, -u^d_{N_d}   \end{bmatrix} \\
  \Theta_N &=  \begin{bmatrix}
        x_0, \cdots, x_{N-1}, & -x_0, \cdots, -x_{N-1} \\
        u_0, \cdots, u_{N-1}, & -u_0, \cdots, -u_{N-1}
    \end{bmatrix}.
    \end{align*}
\end{theorem}
\begin{proof}
Begin by scalarizing the LMIs \eqref{eq:DLMI} by pre/post-multiplying by the indeterminate $z$. Define
\beq  \label{eq:primal}
     p(A, B) \doteq g + z_1^T\begin{bmatrix}0 & AP+ BM \\ * & 0 \end{bmatrix} z_1
\eeq
 with all variables except $(A,B)$  treated as fixed parameters and let
 $p^*= \min_{(A,B)\in \mathcal{C}_N} p(A,B)$.
It is clear that Problem~\ref{prob:robust_lqr} is feasible if and only if the optimal value $p^*>0$ when $z\neq 0$. The corresponding Lagrangian is:
\begin{multline}
L(A,B,\lambda)  =  p(A,B) \\+ \sum \lambda_k^T\left(\begin{bmatrix}
        A & B \\ -A & -B
    \end{bmatrix} \begin{bmatrix}
        x_k \\ u_k
    \end{bmatrix} - \begin{bmatrix}
        \epsilon \textbf{1} + x_{k+1} \\ \epsilon \textbf{1} - x_{k+1}
    \end{bmatrix}\right).
\end{multline}
Collecting all the multipliers $\lambda_k^T$'s into $\Lambda$ and with some algebra, the Lagrangian can be rewritten as 
\begin{multline}
L(A,B,\lambda) = g + \text{Tr}\left( \begin{bmatrix}
        A & B
    \end{bmatrix} \Theta \Lambda - D\Lambda \right) \\
    + \text{Tr}\left( \begin{bmatrix}
        A & B
    \end{bmatrix} 2\begin{bmatrix}
        P \\ M
    \end{bmatrix}\begin{bmatrix}
        I_n & 0_n
    \end{bmatrix} z_1 z_1^T \begin{bmatrix}
        0_n \\ I_n
    \end{bmatrix} \right).
\end{multline}
Thus, the dual problem of \eqref{eq:primal} is:
\begin{subequations}\label{eq:dual}
\begin{align}  
    d^*=\max_{\Lambda_{i,j} \geq 0} \quad &  d(\Lambda)\doteq g - \Tr{D\Lambda}  \label{eq:dual1}\\
    \st \;\quad & \Theta \Lambda = -2\begin{bmatrix}
        P \\ M
    \end{bmatrix}\begin{bmatrix}
        I_n & 0_n
    \end{bmatrix} z_1 z_1^T \begin{bmatrix}
        0_n \\ I_n
    \end{bmatrix} \label{eq:dual2}.
\end{align}
\end{subequations}
Since the primal problem is affine in $(A,B)$ with non-empty $\mathcal{C}$, from Lemma~\ref{lem:duality} strong duality holds. Thus, $p^*>0$ iff $d^*> 0$. Next, we show that the elements of $\Lambda(z)$ can be chosen as continuous functions. To this end, note that for each $z$, \eqref{eq:dual} can be written as a linear program in the entries of $\Lambda$, where $z$ appears only in the right-hand side:
\begin{align} 
  d^*= & \max_{\mu, \Lambda_{i,j} \geq 0}  \mu  \nonumber \\
   & \quad\st \quad  \textrm{\eqref{eq:dual2} and }
  \mu + \Tr{D\Lambda} \leq g(z). \label{eq:dualz}
\end{align}
Continuity of $\Lambda(z)$ follows from  Theorem
2.4 in \cite{Mangasarian} establishing continuity
of the solutions of linear programs with respect to
perturbations in the right-hand side. Finally, the fact that $z$ can be restricted to $\|z\|_2 \leq 1$ follows from the fact that the right-hand side in \eqref{eq:dualz} is homogeneous in $z$.
\end{proof}

\subsection{Polynomial Optimization with SDP Relaxation}
From Theorem~\ref{thm:ddc}, the min-max problem in \eqref{eq:robust_lqr} reduces to the following (infinite-dimensional) polynomial optimization, which can be reformulated as an infinite-dimensional functional linear program over continuous, non-positive multiplier functions $\Lambda_{i,j}(z)$ over the domain $\|z\|_2 \leq 1$. Thus, by Stone-Weierstrass Theorem, each $\Lambda_{i,j}$ can be uniformly approximated, arbitrarily close, with polynomials, leading to the following problem.
\begin{problem} \label{prob:}
Given data samples $\mathcal{D}$
such that int($\mathcal{C}_N$) $\not = \emptyset$, find $\tau, \beta,\gamma > 0$, matrices $P\in\R^{n\times n}\succ 0, M\in\R^{1\times n}, W\in\R$ and a polynomial matrix $\Lambda(z)\in\R [z]^{2(N+N_d)\times n}>0$ in $\|z\|_2\leq 1$ such that:
\beq \label{eq:robust_lqr_poly}
    \min_{P, M, W, \tau, \beta,\gamma} \tau \gamma\quad \st \;
\textrm{Equations \eqref{eq:thm2_1}, \eqref{eq:thm2_2} hold.}\nonumber
\eeq
\end{problem}

 To obtain tractable solutions, finite-dimensional relaxations can be constructed via Putinar's \ac{P-satz}. 

Since $\|z\|_2\leq 1$ defines an Archimedean set, \eqref{eq:robust_lqr_poly} is feasible if and only if there exist \ac{SOS} polynomials $\sigma_i^{(k)} \in\Sigma[x]$ such that the following conditions hold
\begin{align}
\Lambda_{i,j}(z) &= \sigma_0^{(i,j)}(z) + \sigma_1^{(i,j)}(z) (1-z^T z), \label{eq:sos1}\\
g(z)- \Tr{D\Lambda(z)} &= \sigma_0^{(g)}(z) + \sigma_1^{(g)}(z) (1-z^T z), \label{eq:sos2}
\end{align}
where $\Lambda_{i,j}$ denote the entries of the matrix $\Lambda$ and all \ac{SOS} multipliers appearing in different constraints are treated as independent decision variables. 

With \ac{SDP} relaxations (up to certain degree) of each \ac{SOS} function, Problem~\ref{prob:} reduces to an \ac{SDP} that can be efficiently solved using off-the-shelf solvers.  Applying this idea at each time step and updating the consistency set with the new data leads to the data-driven receding-horizon \ac{LQR} control Algorithm \ref{alg:ddc}.

\begin{algorithm}[h]
    \caption{Robust Data-Driven RH LQR}\label{alg:ddc}
\begin{algorithmic}[1]
    \State {Given initial measurements $\mathcal{D}_{train}$, a noise bound $\epsilon$, initial condition $x_1$ and initial $P_0\succ 0,\tau_0>0$}
    \State Start at $k=1$, repeat:
    \begin{align}
    (P_k,M_k,\tau_k) \gets& 
    \mathop{\arg\min} \; \tau \gamma \label{eq:dd_lqr} \\
    \st \quad 
    & \eqref{eq:tau},\eqref{eq:p},\eqref{eq:thm2_1},\eqref{eq:sos1},\eqref{eq:sos2}\text{ hold} \nonumber
    \end{align}
   \State \quad Apply the control $u_k \gets M_kP_k^{-1} x_k$
   \State \quad Collect new measurement $x_{k+1}$
    \State \quad Add $\left\{x_k, u_k, x_{k+1}\right\}$ to the data record 
    \State \quad $k \gets k+1$
   \State Until termination
\end{algorithmic}
\end{algorithm}

\begin{theorem}\label{thm:main} 
    The control law generated by Algorithm \ref{alg:ddc} satisfies the input constraint and asymptotically stabilizes the (unknown) closed-loop system $(A,B)$ that generated the data while minimizing a worst-case upper bound of the \ac{LQR} cost.
\end{theorem}
\begin{proof}
    The proof follows from Theorem~\ref{thm:model_based} by noting that the unknown LTI system is in the consistency set $\mathcal{C}_k$ for all $k$ and thus satisfaction of \eqref{eq:tau},\eqref{eq:p},\eqref{eq:thm2_1},\eqref{eq:sos1},\eqref{eq:sos2} implies that $(A,B)$ satisfies \eqref{eq:DLMI}-\eqref{eq:p}.
\end{proof}

\begin{remark} 
     Leveraging the past execution data  potentially  shrinks  the consistency set, since
     $\mathcal{C}_{k+1} \subseteq
     \mathcal{C}_{k}$. In turn, this can lead to tighter bounds on the worst-case cost \cite{zheng2025receding}.
\end{remark}

\subsection{Complexity}

For a fixed $\tau$, solving problem \eqref{eq:dd_lqr} involves $2(N+N_d)n$ polynomial multipliers in the indeterminate variable $z\in\R^{n_z = 5n+4}$, and each polynomial multiplier involves 2 Gram matrices with relaxation order $r$ and $r-1$ have size of $\binom{n_z+r}{r}$ and $\binom{n_z+r-1}{r-1}$, respectively. Thus, the \ac{SDP} program has nearly $4(N+N_d)n\binom{n_z+r}{r} ^ 2$ variables and using an interior-point method, the total complexity is roughly $\mathcal{O}\left((N+N_d)^3n^{6r+3}\right)$. For comparison, using Scherer's P-satz will involve polynomials in $\mathcal{O}(n(n+1))$ indeterminates
(the elements of $[A, B]$) and thus computational complexity will scale as $\mathcal{O}((N+N_d)^3n^{12r})$.

\section{EXPERIMENTS} \label{sec:experiments}
In this section, we present the simulation results of the proposed robust data-driven receding-horizon input-constrained LQR method. The experiments are conducted in MATLAB 2024a, implemented using YALMIP \cite{Lofberg2004} and solved with MOSEK \cite{mosek}. The code for the experiments is publicly available at 
\url{https://github.com/J-mzz/ddlqr-rh}.

\begin{example}
    Consider a discrete-time linear system 
    \beq\nonumber
    A = \begin{bmatrix}
        0 & -0.99 \\ 0.99 & 0
    \end{bmatrix},\quad  
    B = \begin{bmatrix}
        0 \\ 1
    \end{bmatrix},
    \eeq
    with the control constraint $|u| \leq 1$ and noise level $\epsilon=0.1$.
\end{example}

We initialize the experiments with 5 noisy measurements, randomly generated, and compare the proposed algorithm with the data-driven \ac{LQR} baseline in \cite{dai2023data}. The results in Fig.~\ref{fig:eg1} demonstrate the ability of the proposed method in enforcing input constraints. Specifically,  Fig.~\ref{fig:eg1_traj} illustrates the 10-steps closed-loop trajectories under process noise, with the controller from Algorithm~\ref{alg:ddc} (orange) and the baseline (blue). The corresponding control inputs are shown in Fig.~\ref{fig:eg1_control}. Unlike the baseline, the proposed method strictly enforces the control constraint, ensuring robust feasibility at the cost of slower convergence and a higher upper bound on the cost.

\begin{figure}[ht]
    \centering
    \begin{subfigure}[bt]{0.49\columnwidth}
        \centering
        \includegraphics[width=\linewidth]{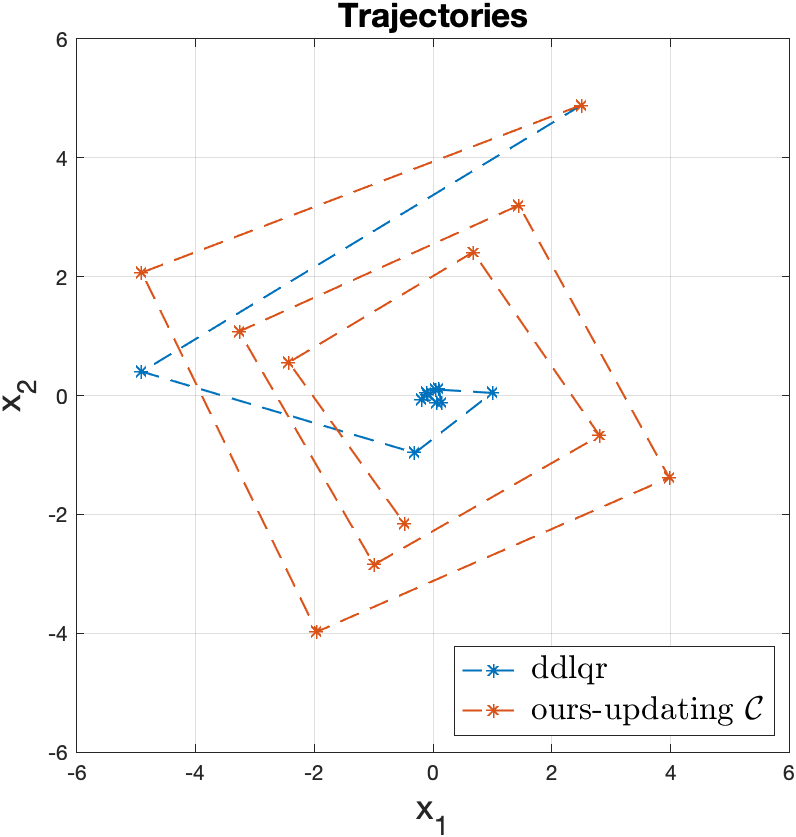}
        \caption{Trajectory}
        \label{fig:eg1_traj}
    \end{subfigure}
    \hfill
    \begin{subfigure}[bt]{0.49\columnwidth}
        \centering
        \includegraphics[width=\linewidth]{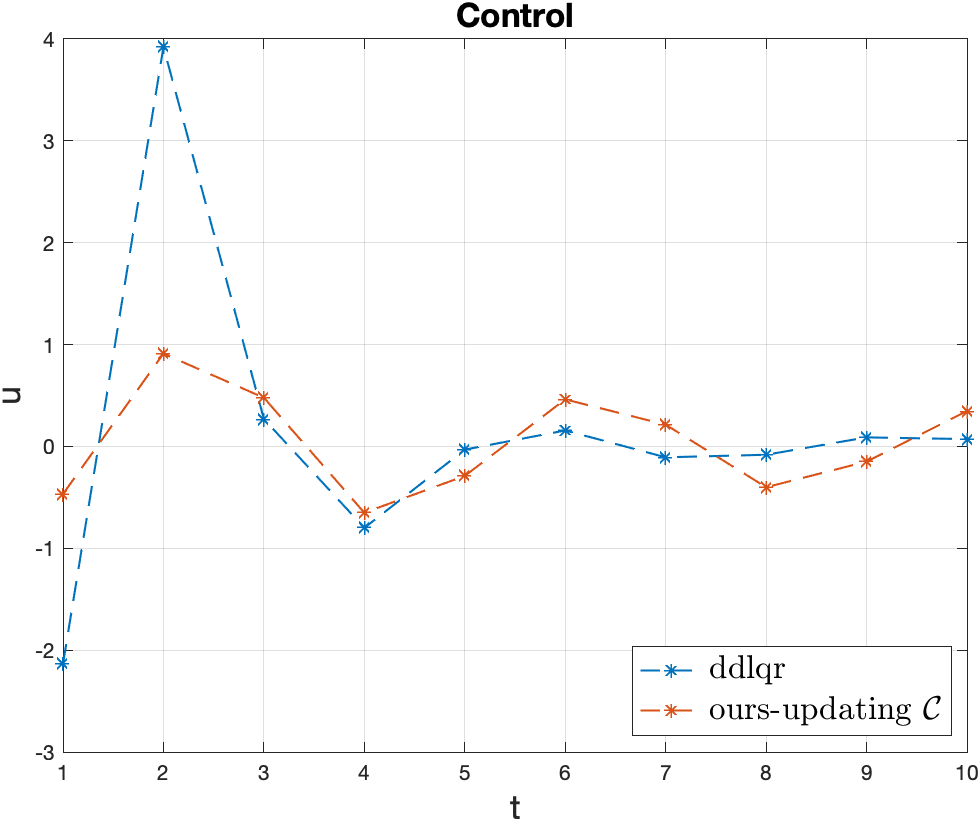}
        \caption{Control}
        \label{fig:eg1_control}
    \end{subfigure}
    \caption{Data-driven LQR with input constraints}
    \label{fig:eg1}
\end{figure}

Further experiments demonstrate the benefits of updating the consistency set with execution data. Specifically, Fig.~\ref{fig:eg1_bound} shows the decrease of the worst-case upper bound along the trajectory as more execution data are incorporated in a receding-horizon manner. The worst-case LQR cost obtained using the proposed method with a fixed consistency set (yellow) and an updated consistency set (orange) are compared, together with the baseline (blue). Fig.~\ref{fig:eg1_consistency} depicts the shrinkage of the consistency set, projected onto the first three PCA components, after incorporating 10 additional executions. These results confirm that receding-horizon updates tighten the upper bound and reduces conservatism.

\begin{figure}[ht]
    \centering
    \begin{subfigure}[bt]{0.49\columnwidth}
        \centering
        \includegraphics[width=\linewidth]{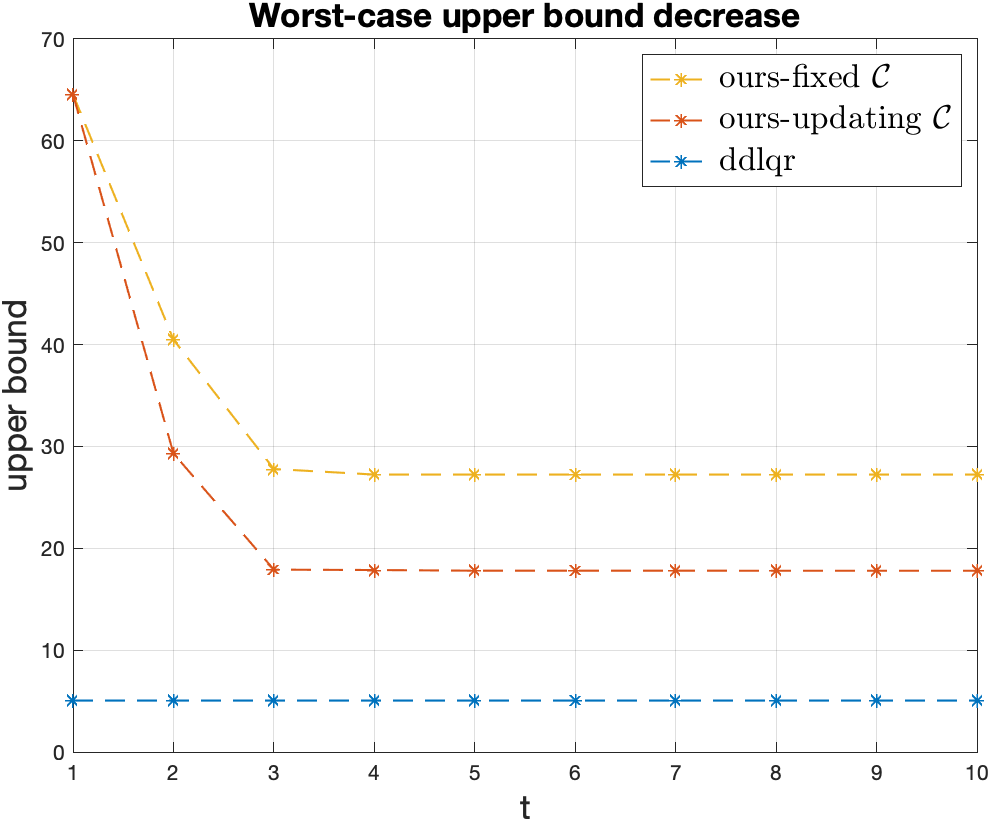}
        \caption{Upper bound}
        \label{fig:eg1_bound}
    \end{subfigure}
    \hfill
    \begin{subfigure}[bt]{0.49\columnwidth}
        \centering
        \includegraphics[width=\linewidth]{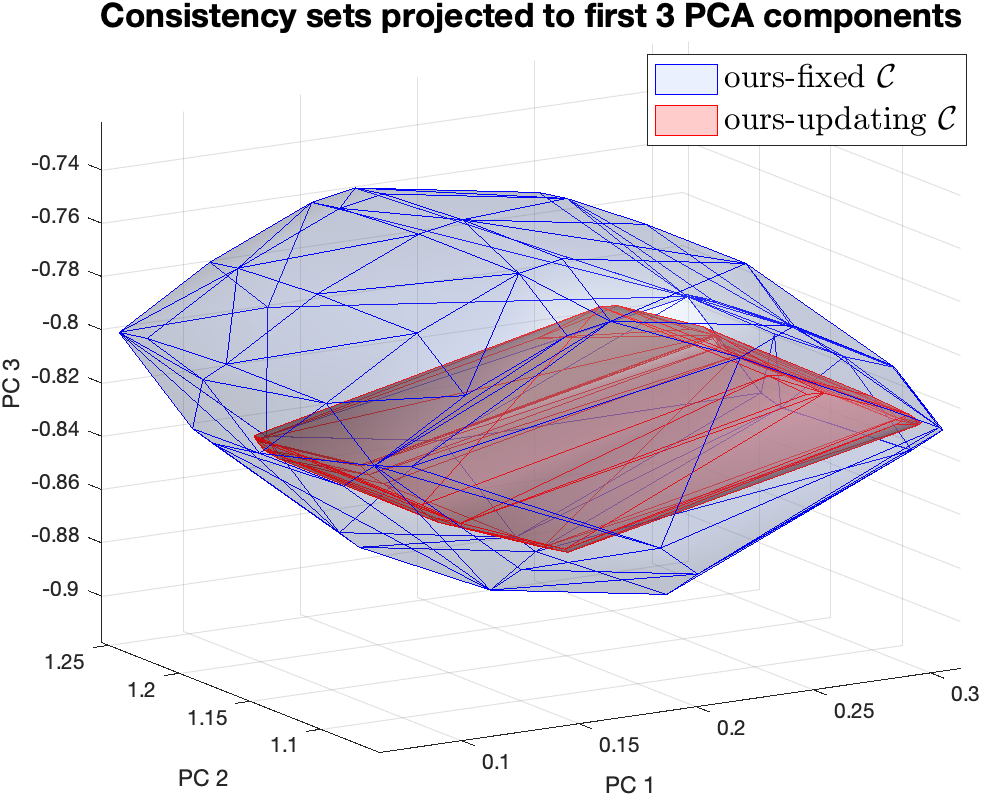}
        \caption{Consistency set}
        \label{fig:eg1_consistency}
    \end{subfigure}
    \caption{Data-driven LQR with receding-horizon}
    \label{fig:eg1_2}
\end{figure}

\section{CONCLUSIONS} \label{sec:conclusions}

This letter presents a robust data-driven \ac{LQR} framework that incorporates input constraints and leverages execution data to tighten the worst-case upper bound. Simulation results illustrate the effectiveness of the proposed algorithms from feasible control actions and the decrease of the bound along the flow. One limitation of the proposed method is the line search on $\tau$, however, it can be efficiently implemented by a proper gridding over the range of $\tau$ with the parallel computing technique. Future work may focus on extending the current method for multi-input systems with polytopic or more general control bound, incorporating state constraints in the framework, or more efficient methods to compute $\tau$.

\addtolength{\textheight}{-12cm}   








\bibliographystyle{IEEEtran} 
\bibliography{reference}



\end{document}